\documentclass[a4paper]{amsart}
\usepackage{amsmath,amsthm,amssymb,latexsym,epic}
\usepackage{graphicx,enumerate}

\newtheorem{theorem}{Theorem}
\newtheorem{lemma}[theorem]{Lemma}

\newtheorem{remark}[theorem]{Remark}
\usepackage[all]{xy}
\begin{document}
\title{Combinatorial Gelfand models for semisimple diagram algebras}
\author{Volodymyr Mazorchuk}
\date{}

\begin{abstract}
We construct combinatorial (involutory) Gelfand models for the following diagram algebras
in the case when they are semi-simple: Brauer algebra, its partial analogue, walled Brauer
algebra, its partial analogue, Temperley-Lieb algebra, its partial analogue, 
walled Temperley-Lieb algebra, its partial analogue, partition algebra and its 
Temperley-Lieb analogue.
\end{abstract}
\maketitle

\section{Introduction and description of the result}\label{s1}

Given an associative algebra $A$, a {\em Gelfand model} for $A$ is an $A$-module which is isomorphic to the
multiplicity-free direct sum of all simple $A$-modules, see \cite{BGG}. If $A$ is finite dimensional and
semi-simple, then each Gelfand model for $A$ is an additive generator of the category $A\text{-}\mathrm{mod}$
and hence completely determines the representation theory of $A$ in some sense. 

Surprisingly enough, it turns out that in many cases there exists a relatively ``easy'' and ``elementary'' 
Gelfand model in comparison to, say, explicit description of all simple $A$-modules. One of the best examples
is the combinatorial Gelfand model for complex representations of the symmetric group $S_n$ defined as follows
(see \cite{APR,IRS,KV}): Let $\mathcal{I}_n$ be the set of involutions in $S_n$, that is the set of all elements
$s\in S_n$ such that $s^{2}=e$, where $e$ is the identity element. For $\pi\in S_n$ and $s\in \mathcal{I}_n$ let
$i(\pi,s)$ denote the number of pairs $(i,j)$, where $1\leq i<j\leq n$, such that 
$s(i)=j$ and $\pi(i)>\pi(j)$. Then the formal vector space $\mathbb{C}[\mathcal{I}_n]$ with basis
$\{v_{s}:s\in \mathcal{I}_n\}$ gets the structure of a Gelfand model for $S_n$ by defining the 
action of $S_n$ on $\mathbb{C}[\mathcal{I}_n]$ via
\begin{equation}\label{eq55}
\pi\cdot v_{s}:=(-1)^{i(\pi,s)}v_{\pi s\pi^{-1}}\quad\quad\text{ for all } \pi\in S_n, s\in \mathcal{I}_n.
\end{equation}
The simplicity of this definition should be compared with the complexity of the classical description of 
the individual constituents of this module, that is the classical
Specht (i.e. simple) $S_n$-modules over $\mathbb{C}$,  see e.g. \cite[Chapter~2]{Sa}. As this Gelfand model
is constructed in purely combinatorial terms on the set of involutions, it is sometimes also called 
an {\em involutary} or {\em combinatorial} Gelfand model.
Similar models can be defined for many other finite groups, in particular, for all classical Weyl groups,
see \cite{APR2,Ar,ABi,AB,Ca,CF,CF2,GO} and references therein.

The paper \cite{KM} makes a step beyond the group theory and constructs Gelfand models for various semigroup
algebras, in particular, for semigroup algebras of inverse semigroups in which all maximal subgroups are
isomorphic to direct sums of symmetric groups. This, of course, admits a straightforward generalization
to inverse semigroups for which Gelfand models are know for all maximal subgroups. The goal of the present note 
is to make another step and construct Gelfand models for various classes of diagram algebras in the semi-simple 
case.

One of these diagram algebra, for which we construct a Gelfand model, is the classical Brauer algebra
from \cite{Br}, which also seems to be one of the oldest diagram algebras. The other algebras are: the
partial analogue of the Brauer algebra, \cite{Ma1,Ma2,MM2}; the walled Brauer algebra, \cite{Tu,Ko},
and its partial analogue; the (walled) Temperley-Lieb algebra, \cite{TL}, and
its partial analogue, \cite{BH,MM2}; the partition algebra, \cite{Jo,Mar}, and its Temperley-Lieb analogue. All these 
algebras depend on a parameter, $\delta\in\mathbb{C}$, and they all are semi-simple for generic values
of this parameter. Furthermore, all these algebras share the following properties: 
\begin{itemize}
\item they have a combinatorial basis given by a set of diagrams;
\item they have an involution $\star$ defined combinatorially on this basis;
\item they have a filtration by ideals for which subquotients ``look like'' some Morita-equivalent versions
of (direct sums of) symmetric groups.
\end{itemize}

The main idea of the paper is that instead of the set of involutions one should consider the set of diagrams, 
which are self-dual with respect to $\star$. The formal vector space generated by this set admits
the natural module structure given by ``conjugation'', were the action on the right is defined using $\star$,
and in which one should take into account two types of scalars, those coming from the definition of the
diagram algebra, and those coming from the combinatorial Gelfand model for $S_n$ as adjusted in \cite{KM}.
One essential difference with \cite{KM} is that the set of $\star$-self-dual diagrams is much smaller
than the set of all involutions in all maximal subgroups which was considered in \cite{KM} (however,
for inverse semigroups, that is the setup of \cite{KM}, these two sets coincide). Our approach heavily 
exploits the combinatorics of the classical representation theory of finite semigroups, see \cite{GM,GMS}.

The main result of the paper, which gives an explicit combinatorial Gelfand model for all diagram
algebras mentioned above in the semi-simple case, is formulated and proved in Section~\ref{s3} after 
preliminaries on diagram algebras and their representations that are collected in Section~\ref{s2}.
For the Brauer algebra this Gelfand model was considered in \cite{Sch}.
\vspace{2mm}

\noindent
{\bf Acknowledgements.} The research was partially supported by the Royal Swedish Academy of Sciences
and the Swedish Research Council.

\section{Classical diagram algebras}\label{s2}

\subsection{Partitions and diagrams}\label{s2.1}

In this paper we work over $\mathbb{C}$. Fix a positive integer $n$ and consider the sets
$\underline{n}:=\{1,2,\dots,n\}$, $\underline{n}':=\{1',2',\dots,n'\}$ and $\mathbf{n}:=\underline{n}\cup
\underline{n}'$ (the union is automatically disjoint). 
Denote by $\mathtt{P}_n$ the set of all equivalence relations on $\mathbf{n}$, also
known as {\em partitions} of $\mathbf{n}$. For $\rho\in \mathtt{P}_n$ an equivalence class of
the equivalence relation $\rho$ is usually called a {\em part} of $\rho$ (which corresponds to viewing
$\rho$ as a partition). Define a binary operation 
$\circ$ on $\mathtt{P}_n$, called {\em composition}, as follows:
given $\tau,\rho\in \mathtt{P}_n$ define $\tau\circ\rho\in \mathtt{P}_n$ via the following procedure:
\begin{itemize}
\item First consider $\tau$ and $\rho$ as partitions of disjoint sets 
\begin{displaymath}
\{1_{\tau},2_{\tau},\dots, n_{\tau},1'_{\tau},2'_{\tau},\dots, n'_{\tau}\}\quad\text{ and }\quad 
\{1_{\rho},2_{\rho},\dots, n_{\rho},1'_{\rho},2'_{\rho},\dots, n'_{\rho}\}, 
\end{displaymath}
respectively.
\item Identify $i_{\tau}$ with $i'_{\rho}$ for all $i\in \underline{n}$.
\item Define $\chi$ as the minimal equivalence relation on the set 
\begin{displaymath}
\{1_{\rho},2_{\rho},\dots, n_{\rho},
1'_{\rho}=1_{\tau},2'_{\rho}=2_{\tau},\dots, n'_{\rho}=n_{\tau},
1'_{\tau},2'_{\tau},\dots, n'_{\tau}\} 
\end{displaymath}
which contains both $\tau$ and $\rho$.
\item Define $\tau\circ\rho$ as the restriction of $\chi$ to 
$\{1_{\rho},2_{\rho},\dots, n_{\rho},1'_{\tau},2'_{\tau},\dots, n'_{\tau}\}$ and identify the latter
set with $\mathbf{n}$ by removing all subscripts.
\end{itemize}
Denote also by $c(\tau,\rho)$ the number of equivalence classes of $\chi$ which are subsets of the set
$\{1'_{\rho}=1_{\tau},2'_{\rho}=2_{\tau},\dots, n'_{\rho}=n_{\tau}\}$.

A partition $\rho\in \mathtt{P}_n$ is usually depicted as a diagram as shown in Figure~\ref{fig1}.
An example of composition of two partitions is given in Figure~\ref{fig2}.
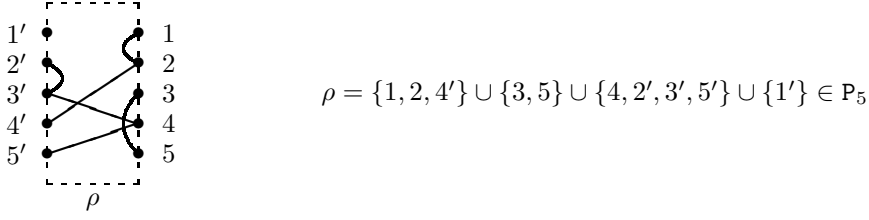
\begin{figure}
\special{em:linewidth 0.4pt} \unitlength 0.80mm
\begin{picture}(200.00,45.00)
%%%%%%%%%%%%%%%%%%%%%%%%%%%%%%%%%%%%%%%%%%%%%%%%%%%%%%%%
\put(10.00,10.00){\makebox(0,0)[cc]{$\bullet$}}
\put(10.00,15.00){\makebox(0,0)[cc]{$\bullet$}}
\put(10.00,20.00){\makebox(0,0)[cc]{$\bullet$}}
\put(10.00,25.00){\makebox(0,0)[cc]{$\bullet$}}
\put(10.00,30.00){\makebox(0,0)[cc]{$\bullet$}}
\put(25.00,10.00){\makebox(0,0)[cc]{$\bullet$}}
\put(25.00,15.00){\makebox(0,0)[cc]{$\bullet$}}
\put(25.00,20.00){\makebox(0,0)[cc]{$\bullet$}}
\put(25.00,25.00){\makebox(0,0)[cc]{$\bullet$}}
\put(25.00,30.00){\makebox(0,0)[cc]{$\bullet$}}
%%%%%%%%%%%%%%%%%%%%%%%%%%%%%%%%%%%%%%%%%%%%%%%%%%
\dashline{1}(10.00,05.00)(10.00,35.00)
\dashline{1}(25.00,35.00)(10.00,35.00)
\dashline{1}(25.00,35.00)(25.00,05.00)
\dashline{1}(10.00,05.00)(25.00,05.00)
\thicklines
\drawline(10.00,15.00)(25.00,25.00)
\drawline(10.00,10.00)(25.00,15.00)
\drawline(10.00,20.00)(25.00,15.00)
\qbezier(10,20)(15,22.50)(10,25)
\qbezier(25,25)(20,27.50)(25,30)
\qbezier(25,10)(20,15)(25,20)
%%%%%%%%%%%%%%%%%%%%%%%%%%%%%%%%%%%%%%%%%%%%%%%%%%
\put(17.50,02.00){\makebox(0,0)[cc]{$\rho$}}
\put(100,20){\makebox(0,0)[cc]{$\rho=\{1,2,4'\}\cup\{3,5\}\cup\{4,2',3',5'\}\cup\{1'\}\in\mathtt{P}_5$}}
\put(30,10){\makebox(0,0)[cc]{$5$}}
\put(30,15){\makebox(0,0)[cc]{$4$}}
\put(30,20){\makebox(0,0)[cc]{$3$}}
\put(30,25){\makebox(0,0)[cc]{$2$}}
\put(30,30){\makebox(0,0)[cc]{$1$}}
\put(5,10){\makebox(0,0)[cc]{$5'$}}
\put(5,15){\makebox(0,0)[cc]{$4'$}}
\put(5,20){\makebox(0,0)[cc]{$3'$}}
\put(5,25){\makebox(0,0)[cc]{$2'$}}
\put(5,30){\makebox(0,0)[cc]{$1'$}}
\end{picture}
%}}}
\caption{The diagram of a partition}
\label{fig1}
\end{figure}
\begin{figure}
\special{em:linewidth 0.4pt} \unitlength 0.80mm
\begin{picture}(90.00,45.00)
%%%%%%%%%%%%%%%%%%%%%%%%%%%%%%%%%%%%%%%%%%%%%%%%%%%%%%%%
\put(10.00,10.00){\makebox(0,0)[cc]{$\bullet$}}
\put(10.00,15.00){\makebox(0,0)[cc]{$\bullet$}}
\put(10.00,20.00){\makebox(0,0)[cc]{$\bullet$}}
\put(10.00,25.00){\makebox(0,0)[cc]{$\bullet$}}
\put(10.00,30.00){\makebox(0,0)[cc]{$\bullet$}}
\put(25.00,10.00){\makebox(0,0)[cc]{$\bullet$}}
\put(25.00,15.00){\makebox(0,0)[cc]{$\bullet$}}
\put(25.00,20.00){\makebox(0,0)[cc]{$\bullet$}}
\put(25.00,25.00){\makebox(0,0)[cc]{$\bullet$}}
\put(25.00,30.00){\makebox(0,0)[cc]{$\bullet$}}
\put(50.00,10.00){\makebox(0,0)[cc]{$\bullet$}}
\put(50.00,15.00){\makebox(0,0)[cc]{$\bullet$}}
\put(50.00,20.00){\makebox(0,0)[cc]{$\bullet$}}
\put(50.00,25.00){\makebox(0,0)[cc]{$\bullet$}}
\put(50.00,30.00){\makebox(0,0)[cc]{$\bullet$}}
\put(35.00,10.00){\makebox(0,0)[cc]{$\bullet$}}
\put(35.00,15.00){\makebox(0,0)[cc]{$\bullet$}}
\put(35.00,20.00){\makebox(0,0)[cc]{$\bullet$}}
\put(35.00,25.00){\makebox(0,0)[cc]{$\bullet$}}
\put(35.00,30.00){\makebox(0,0)[cc]{$\bullet$}}
\put(70.00,10.00){\makebox(0,0)[cc]{$\bullet$}}
\put(70.00,15.00){\makebox(0,0)[cc]{$\bullet$}}
\put(70.00,20.00){\makebox(0,0)[cc]{$\bullet$}}
\put(70.00,25.00){\makebox(0,0)[cc]{$\bullet$}}
\put(70.00,30.00){\makebox(0,0)[cc]{$\bullet$}}
\put(85.00,10.00){\makebox(0,0)[cc]{$\bullet$}}
\put(85.00,15.00){\makebox(0,0)[cc]{$\bullet$}}
\put(85.00,20.00){\makebox(0,0)[cc]{$\bullet$}}
\put(85.00,25.00){\makebox(0,0)[cc]{$\bullet$}}
\put(85.00,30.00){\makebox(0,0)[cc]{$\bullet$}}
%%%%%%%%%%%%%%%%%%%%%%%%%%%%%%%%%%%%%%%%%%%%%%%%%%
\dashline{1}(10.00,05.00)(10.00,35.00)
\dashline{1}(25.00,35.00)(10.00,35.00)
\dashline{1}(25.00,35.00)(25.00,05.00)
\dashline{1}(10.00,05.00)(25.00,05.00)
\dashline{1}(35.00,05.00)(35.00,35.00)
\dashline{1}(50.00,35.00)(35.00,35.00)
\dashline{1}(50.00,35.00)(50.00,05.00)
\dashline{1}(35.00,05.00)(50.00,05.00)
\dashline{1}(70.00,05.00)(70.00,35.00)
\dashline{1}(85.00,35.00)(70.00,35.00)
\dashline{1}(85.00,35.00)(85.00,05.00)
\dashline{1}(70.00,05.00)(85.00,05.00)
%%%%%%%%%%%%%%%%%%%%%%%%%%%%%%%%%%%%%%%%%%%%%%%%%%
\thicklines
\drawline(10.00,30.00)(25.00,25.00)
\drawline(10.00,20.00)(25.00,25.00)
\drawline(10.00,10.00)(25.00,15.00)
\qbezier(10,15)(15,17.50)(10,20)
\qbezier(25,15)(20,12.50)(25,10)
\qbezier(25,20)(20,25)(25,30)
\drawline(35.00,25.00)(50.00,30.00)
\drawline(35.00,10.00)(50.00,20.00)
\qbezier(35,15)(40,20)(35,25)
\qbezier(35,20)(40,25)(35,30)
\qbezier(50,15)(45,20)(50,25)
\drawline(70.00,30.00)(85.00,30.00)
\drawline(70.00,20.00)(85.00,20.00)
\qbezier(85,15)(80,20)(85,25)
\qbezier(70,20)(75,25)(70,30)
\qbezier(70,15)(75,17.50)(70,20)
\qbezier(70,15)(75,12.50)(70,10)
%%%%%%%%%%%%%%%%%%%%%%%%%%%%%%%%%%%%%%%%%%%%%%%%%%
\put(30,20){\makebox(0,0)[cc]{$\circ$}}
\put(60,20){\makebox(0,0)[cc]{$=$}}
\put(17.50,02.00){\makebox(0,0)[cc]{$\tau$}}
\put(42.50,02.00){\makebox(0,0)[cc]{$\rho$}}
\put(77.50,02.00){\makebox(0,0)[cc]{$\rho\circ\tau$}}
\end{picture}
%}}}
\caption{Composition of partitions, $c(\tau,\rho)=1$}
\label{fig2}
\end{figure}
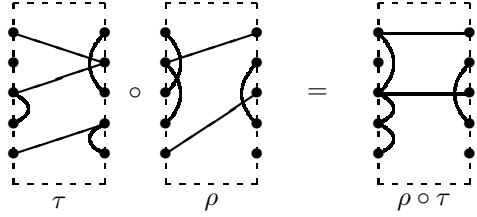

For $\rho\in\mathtt{P}_n$ the {\em rank} of $\rho$, denoted $\mathbf{r}(\rho)$, is the number
of parts in $\rho$ which intersect both $\underline{n}$ and $\underline{n}'$ non-trivially.
Parts with this property are called {\em propagating lines}. A part of $\rho$ is called a 
{\em right} or {\em left} part if it is contained in $\underline{n}$ or $\underline{n}'$, respectively.
Each part of $\rho$ is either a left or right part or a propagating line.
We have $0\leq \mathbf{r}(\rho)\leq n$.

\subsection{Diagram algebras}\label{s2.2}

For a fixed $\delta\in\mathbb{C}$ (which is always assumed to be nonzero in this paper), 
the formal linear span $\mathbf{P}_n$ of $\mathtt{P}_n$ has the natural 
structure of an associative algebra defined for the basis elements via $\tau\rho:=\delta^{c(\tau,\rho)}\tau\circ\rho$
and extended to the whole of $\mathbf{P}_n$ by bilinearity. This algebra $\mathbf{P}_n=\mathbf{P}_n(\delta)$ 
is called the {\em partition algebra}, see \cite{Jo,Mar}. Other diagram algebras are subalgebras of $\mathbf{P}_n$
constructed as linear spans of certain sets of diagrams below.

Let $\mathtt{B}_n$ and $\mathcal{P}\mathtt{B}_n$ denote the sets of all $\rho\in\mathtt{P}_n$ such that all parts
of $\rho$ have cardinality two or at most two, respectively. Let $\mathbf{B}_n$ and $\mathcal{P}\mathbf{B}_n$ be
the linear spans of $\mathtt{B}_n$ and $\mathcal{P}\mathtt{B}_n$ inside $\mathbf{P}_n$, respectively. It is
easy to see that  both $\mathbf{B}_n$ and $\mathcal{P}\mathbf{B}_n$ are subalgebras of $\mathbf{P}_n$ called
the {\em Brauer algebra} (see \cite{Br}) and the {\em partial Brauer algebra} (see \cite{Ma1,Ma2,MM2}), respectively.
Note that, if $\rho\in \mathtt{B}_n$, then $n$ and $\mathbf{r}(\rho)$ have the  same parity. 

Next, let $\mathcal{Q}\mathtt{B}_n$, $\mathcal{QP}\mathtt{B}_n$ and $\mathcal{Q}\mathtt{P}_n$ denote the sets
of planar diagrams in $\mathtt{B}_n$, $\mathcal{P}\mathtt{B}_n$ and $\mathtt{P}_n$, respectively. The linear 
spans $\mathbf{TL}_n$, $\mathcal{P}\mathbf{TL}_n$ and $\mathbf{PTL}_n$ of 
$\mathcal{Q}\mathtt{B}_n$, $\mathcal{QP}\mathtt{B}_n$ and $\mathcal{Q}\mathtt{P}_n$, respectively, 
are also subalgebras of $\mathbf{P}_n$ called the {\em Temperley-Lieb algebra} (see \cite{TL}), 
{\em partial Temperley-Lieb algebra} (see \cite{MM1}, also called {\em Motzkin algebra} in 
\cite{BH}) and {\em Temperley-Lieb-partition algebra} (see \cite{Mar2}), respectively.

Finally, let $a$ and $b$ be two positive integers such that $a+b=n$. Denote by $\mathcal{W}_{a,b}\mathtt{B}_n$,
$\mathcal{W}_{a,b}\mathcal{P}\mathtt{B}_n$, $\mathcal{W}_{a,b}\mathcal{Q}\mathtt{B}_n$ and 
$\mathcal{W}_{a,b}\mathcal{QP}\mathtt{B}_n$ the subsets of $\mathtt{B}_n$, $\mathcal{P}\mathtt{B}_n$,
$\mathcal{Q}\mathtt{B}_n$ and $\mathcal{QP}\mathtt{B}_n$, respectively, consisting of all diagrams which
satisfy the following conditions:
\begin{itemize}
\item each propagating line is either contained in $\mathbf{a}$ or in $\mathbf{n}\setminus\mathbf{a}$;
\item each non-singleton right part intersects both $\underline{a}$ and $\underline{n}\setminus\underline{a}$;
\item each non-singleton left part intersects both $\underline{a}'$ and $\underline{n}'\setminus\underline{a}'$.
\end{itemize}
The linear 
spans $\mathcal{W}_{a,b}\mathbf{B}_n$, $\mathcal{W}_{a,b}\mathcal{P}\mathbf{B}_n$, 
$\mathcal{W}_{a,b}\mathbf{TL}_n$ and $\mathcal{W}_{a,b}\mathcal{P}\mathbf{TL}_n$ of $\mathcal{W}_{a,b}\mathtt{B}_n$,
$\mathcal{W}_{a,b}\mathcal{P}\mathtt{B}_n$, $\mathcal{W}_{a,b}\mathcal{Q}\mathtt{B}_n$ and 
$\mathcal{W}_{a,b}\mathcal{QP}\mathtt{B}_n$, respectively, are subalgebras of $\mathbf{P}_n$ called the
{\em walled Brauer algebra} (see \cite{Tu,Ko}), its partial analogue, {\em walled Temperley-Lieb algebra}
and its partial analogue, respectively.

The algebra $\mathbf{P}_n$ has a natural anti-involution $\star$ given by swapping $\underline{n}$
and $\underline{n}'$. Diagrammatically this corresponds to reflection of a diagram with respect to the vertical
axes. This anti-involution restricts to all subalgebras introduced above.

\subsection{Specht modules}\label{s2.3}

Let $\mathbf{A}$ denote one of the algebras $\mathbf{P}_n$, $\mathbf{B}_n$, 
$\mathcal{P}\mathbf{B}_n$, $\mathbf{TL}_n$, $\mathcal{P}\mathbf{TL}_n$, $\mathbf{PTL}_n$, 
$\mathcal{W}_{a,b}\mathbf{B}_n$, $\mathcal{W}_{a,b}\mathcal{P}\mathbf{B}_n$ 
$\mathcal{W}_{a,b}\mathbf{TL}_n$ or $\mathcal{W}_{a,b}\mathcal{P}\mathbf{TL}_n$.
Denote by $\mathtt{A}$ the diagram basis of $\mathbf{A}$. 

Two elements $\tau,\rho\in\mathtt{A}$ are called
\begin{itemize}
\item {\em left equivalent}, denoted $\tau\sim_{L}\rho$, provided that 
their restrictions to $\underline{n}$ coincide and they have the same right parts;
\item {\em right equivalent}, denoted $\tau\sim_{R}\rho$, provided that their restrictions to $\underline{n}'$ 
coincide and they have the same left parts.
\end{itemize}
We have $\tau\sim_{L}\rho$ if and only if $\mathbf{A}\tau=\mathbf{A}\rho$. Similarly,
$\tau\sim_{R}\rho$ if and only if $\tau\mathbf{A}=\rho\mathbf{A}$. For $\tau,\rho\in\mathtt{A}$
write $\tau\sim_{J}\rho$ provided that $\mathbf{A}\tau\mathbf{A}=\mathbf{A}\rho\mathbf{A}$.

For $\rho\in \mathcal{W}_{a,b}\mathcal{P}\mathtt{B}_n$ denote by
$\mathbf{r}_a(\rho)$ the number of propagating lines of $\rho$ contained in $\mathbf{n}_a$,
the so-called {\em $a$-rank} of $\rho$. The following lemma is straightforward and left to the reader:

\begin{lemma}\label{lem351}
\hspace{1mm}

\begin{enumerate}[$($a$)$]
\item\label{lem351.1} If $\mathbf{A}\neq \mathcal{W}_{a,b}\mathcal{P}\mathbf{B}_n,\mathcal{W}_{a,b}\mathcal{P}\mathbf{TL}_n$,
then $\tau\sim_{J}\rho$ is equivalent to $\mathbf{r}(\tau)=\mathbf{r}(\rho)$.
\item\label{lem351.2} If $\mathbf{A}=\mathcal{W}_{a,b}\mathcal{P}\mathbf{B}_n$ or $\mathbf{A}=\mathcal{W}_{a,b}\mathcal{P}\mathbf{TL}_n$,
then $\tau\sim_{J}\rho$ is equivalent to $\mathbf{r}(\tau)=\mathbf{r}(\rho)$
and $\mathbf{r}_a(\tau)=\mathbf{r}_a(\rho)$.
\end{enumerate}
\end{lemma}

Clearly, both relations $\sim_{L}$ and $\sim_{R}$ are subsets of the relation $\sim_{J}$.
Furthermore, $\sim_{J}$ is the minimal equivalence relation containing both $\sim_{L}$ and $\sim_{R}$,
see e.g. \cite[Section~4.4]{GM}. Denote by $\mathrm{spec}(\mathbf{A})$ the set of $J$-equivalence classes of diagrams.

Let $\mathcal{L}$ be a left equivalence class in $\mathtt{A}$ and let $\mathcal{J}\in\mathrm{spec}(\mathbf{A})$ 
be the $J$-class containing $\mathcal{L}$. Then the formal linear span $\mathbb{C}[\mathcal{L}]$, considered
as a subquotient of the left regular $\mathbf{A}$-module, inherits the natural structure of an 
$\mathbf{A}$-module given, for $\tau\in \mathtt{A}$  and $\rho\in \mathcal{L}$, by 
\begin{displaymath}
\tau\cdot\rho:=
\begin{cases}
\delta^{c(\tau,\rho)}\tau\circ\rho,& \tau\circ\rho\in \mathcal{L};\\
0,& \text{otherwise}.
\end{cases}
\end{displaymath}
The module $\mathbb{C}[\mathcal{L}]$ does not depend, up to isomorphism, on the choice of 
$\mathcal{L}\subset \mathcal{J}$ (this is similar to e.g. \cite[Proposition~9]{MM2}). 

Note that all elements of $\mathcal{L}$ have both the same rank and the same $a$-rank (the latter in the
case of walled algebras), which we denote by $\mathbf{r}(\mathcal{L})$ and $\mathbf{r}_a(\mathcal{L})$, respectively.
Consider the group
\begin{displaymath}
G=G_{\mathcal{J}}:=
\begin{cases}
S_{\mathbf{r}_a(\mathcal{L})}\oplus S_{\mathbf{r}(\mathcal{L})-\mathbf{r}_a(\mathcal{L})},& 
\mathbf{A}\in\{\mathcal{W}_{a,b}\mathbf{B}_n,\mathcal{W}_{a,b}\mathcal{P}\mathbf{B}_n\};\\
S_{\mathbf{r}(\mathcal{L})},& \mathbf{A}\in\{\mathbf{P}_n,\mathbf{B}_n,\mathcal{P}\mathbf{B}_n\};\\
\{e\},& \text{otherwise}.
\end{cases}
\end{displaymath}
There is an obvious right action of $G$ on the $\mathbf{A}$-module $\mathbb{C}[\mathcal{L}]$ 
by automorphisms which permute the
right ends of the propagating lines. This makes $\mathbb{C}[\mathcal{L}]$ into an 
$\mathbf{A}\text{-}\mathbb{C}[G]$-bimodule which is free over $\mathbb{C}[G]$. The rank of the latter
free module is the number $q_{\mathcal{J}}$ of right equivalence classes in $\mathcal{J}$
(which is also equal to the number of left equivalence classes because of $\star$). 
If $N$ is a simple $G$-module, we can consider the corresponding {\em Specht} $A$-module 
$\Delta_{\mathcal{L}}(N):=\mathbb{C}[\mathcal{L}]\otimes_{\mathbb{C}[G]}N$.
Note that such simple modules $N$ are described by 
\begin{itemize}
\item pairs $(\lambda_1,\lambda_2)$ of partitions $\lambda_1\vdash \mathbf{r}_a(\mathcal{L})$
and $\lambda_2\vdash \mathbf{r}(\mathcal{L})-\mathbf{r}_a(\mathcal{L})$ in the case
$\mathbf{A}\in\{\mathcal{W}_{a,b}\mathbf{B}_n,\mathcal{W}_{a,b}\mathcal{P}\mathbf{B}_n\}$;
\item partitions $\lambda\vdash \mathbf{r}(\mathcal{L})$ in the case
$\mathbf{A}\in\{\mathbf{P}_n,\mathbf{B}_n,\mathcal{P}\mathbf{B}_n\}$;
\item the trivial partition $\lambda\vdash 1$ in all Temperley-Lieb cases.
\end{itemize}
We have $\dim \Delta_{\mathcal{L}}(N)=q_{\mathcal{J}}\dim N$.
For each $\mathcal{J}\in\mathrm{spec}(\mathbf{A})$ fix 
some left equivalence class $\mathcal{L}_{\mathcal{J}}$ contained in $\mathcal{J}$.

\begin{theorem}\label{thm1}
\hspace{1mm}

\begin{enumerate}[$($a$)$]
\item\label{thm1.1} The module $\Delta_{\mathcal{L}_{\mathcal{J}}}(N)$ 
is simple for all but finitely many values of $\delta$. 
\item\label{thm1.2} $\mathbf{A}$ is semi-simple if and only if all $\Delta_{\mathcal{L}_{\mathcal{J}}}(N)$ are simple.
\item\label{thm1.3} If $\mathbf{A}$ is semi-simple, then the set of all $\Delta_{\mathcal{L}_{\mathcal{J}}}(N)$, 
where ${\mathcal{J}}\in\mathrm{spec}(\mathbf{A})$ and $N$ runs through the set of isoclasses of simple 
$G_{\mathcal{J}}$-modules, forms a complete and irredundant set of pairwise non-isomorphic 
simple $\mathbf{A}$-modules.
\end{enumerate}
\end{theorem}

\subsection{Sketch of the proof of Theorem~\ref{thm1}}\label{s2.4}

In most cases, Theorem~\ref{thm1} can be found, sometimes disguised, in the existing literature.
In other cases it is proved using the same arguments as in the known cases. Here we sketch the
standard approach to its proof.

The fact that $\mathtt{A}$ is a union of left equivalence classes means that the left regular 
representation of $\mathbf{A}$ has a filtration whose subquotients are isomorphic to Specht modules.
The endomorphism algebra of $\Delta_{\mathcal{L}}(N)$ is isomorphic to the endomorphism algebra of 
the $G$-module $N$ which, in turn, is isomorphic to $\mathbb{C}$. This implies that 
$\mathtt{A}$ is quasi-hereditary (note that we assume $\delta\neq 0$), with Specht modules being the
standard modules for the quasi-hereditary structure. Moreover, $\mathtt{A}$ has a simple preserving
duality induced by $\star$. Now from the BGG reciprocity it follows that $\mathtt{A}$ is semi-simple if
and only if all $\Delta_{\mathcal{L}}(N)$ are simple (which is claim \eqref{thm1.2}).
Since Specht modules are standard modules, claim \eqref{thm1.2} implies claim \eqref{thm1.3}.

It remains to address claim \eqref{thm1.1}. For the Temperley-Lieb algebra, partition algebra and its 
Temperley-Lieb analogue this claim can be found in \cite{Mar2}.  For the Brauer algebra, see \cite{Ru} 
and references therein. For the walled Brauer algebra, see \cite{Ha,Ni}, and the same argument works 
for the walled Temperley-Lieb algebra. For the partial Brauer algebra, see \cite{MM2}. The argument 
from \cite{MM2} restricts to all other partial algebras as well (i.e. for partial Temperley-Lieb algebras, 
partial walled Brauer algebras and partial walled Temperley-Lieb algebras). This completes the proof.

\section{The main result}\label{s3}

\subsection{The model}\label{s3.1}

Let $\mathbf{A}$ denote one of the algebras $\mathbf{P}_n$, $\mathbf{B}_n$, 
$\mathcal{P}\mathbf{B}_n$, $\mathbf{TL}_n$, $\mathcal{P}\mathbf{TL}_n$,  $\mathbf{PTL}_n$, 
$\mathcal{W}_{a,b}\mathbf{B}_n$, $\mathcal{W}_{a,b}\mathcal{P}\mathbf{B}_n$ 
$\mathcal{W}_{a,b}\mathbf{TL}_n$ and $\mathcal{W}_{a,b}\mathcal{P}\mathbf{TL}_n$.
Denote by $\mathtt{A}$ the diagram basis of $\mathbf{A}$ and 
let $\mathcal{I}$ denote the set of all $\star$-self-dual elements in $\mathtt{A}$.

Let $\tau\in \mathtt{A}$ be a diagram of rank $k\in\{0,1,\dots,n\}$. We associate to $\tau$ an
element $\pi_{\tau}\in S_k$ in the following way: Let $A_1,A_2,\dots,A_k$ be the list of propagating
lines in $\tau$. For $i\in\{1,2,\dots,k\}$ let $B_i$ and $C_i$ denote the intersections of $A_i$ with
$\underline{n}$ and $\underline{n}'$, respectively. We assume that for all $1\leq i < j \leq k$
we have $\min\{s\in B_i\}<\min\{s\in B_j\}$. We now define $\pi_{\tau}\in S_k$ as the unique 
permutation such that for all $1\leq i < j \leq k$
we have $\min\{s\in C_{\pi_{\tau}(i)}\}<\min\{s\in C_{\pi_{\tau}(j)}\}$. For example, for the element
$\rho$ in Figure~\ref{fig1} we have $k=2$, $A_1=\{1,2,4'\}$, $A_2=\{4,2',3',5'\}$, $B_1=\{1,2\}$,
$B_2=\{4\}$, $C_1=\{4'\}$, $C_2=\{2',3',5'\}$ and hence $\pi_{\rho}\in S_2$ is the transposition 
swapping $1$ and $2$. For the element $\tau$ in Figure~\ref{fig2} we similarly get that 
$\pi_{\tau}\in S_2$ is the identity map. For $\iota\in \mathcal{I}$ we note that  
$\{i,j'\}$ belongs to a propagating line if and only if $\{i',j\}$ does. This implies that for
$\iota\in \mathcal{I}$ the element $\pi_{\iota}\in S_2$ is an involution. 
If $\mathbf{A}$ is one of the walled algebras, then, by construction, the element 
$\pi_{\tau}$ belongs to  $S_{\mathbf{r}_a(\tau)}\oplus S_{k-\mathbf{r}_a(\tau)}$.

Consider the formal span $\mathbb{C}[\mathcal{I}]$ with basis
$\{v_{\iota}:\iota\in \mathcal{I}\}$ and for $\tau\in \mathtt{A}$ and $\iota\in \mathcal{I}$ set
\begin{equation}\label{eq1}
\tau\cdot v_{\iota}:=
\begin{cases}
(-1)^{i(\pi_{\tau\circ \iota}\pi_{\iota},\pi_{\iota})}\delta^{c(\tau,\iota)}v_{\tau\circ\iota\circ\tau^{\star}},
& \mathbf{r}(\tau\circ\iota)=\mathbf{r}(\iota);\\
0,& \text{otherwise}.
\end{cases} 
\end{equation}
Here $i(\pi_{\tau\circ\iota}\pi_{\iota},\pi_{\iota})$ is defined as in Section~\ref{s1}.
Note that $\mathbf{r}(\tau\circ\iota)=\mathbf{r}(\iota)$ also implies
$\mathbf{r}(\tau\circ\iota\circ\tau^{\star})=\mathbf{r}(\iota)$ by symmetry.
The explanation for the ``complicated-looking'' element $\pi_{\tau\circ\iota}\pi_{\iota}$ is the following: 
for the formula to make sense we need an element ``just like $\tau$ but of the same rank as $\iota$'' which would
act on propagating lines of $\iota$, the  easiest way to produce such an element is to multiply 
$\tau$ with $\iota$, however, $\iota$ is not idempotent and hence
one has to compensate by multiplying with $\pi_{\iota}=\pi_{\iota}^{-1}$ (recall that $\pi_{\iota}$ is an involution).
We can now state our main result:

\begin{theorem}\label{tmain}
If $\mathbf{A}$ is semi-simple, then formula \eqref{eq1} defines on $\mathbb{C}[\mathcal{I}]$ 
the structure of a Gelfand model for $\mathbf{A}$.
\end{theorem}

\subsection{Proof of Theorem~\ref{tmain}}\label{s3.2}

Theorem~\ref{tmain} follows from the following two lemmata:

\begin{lemma}\label{lem01}
Formula \eqref{eq1} defines on $\mathbb{C}[\mathcal{I}]$ 
the structure of an $\mathbf{A}$-module.  
\end{lemma}

\begin{proof}
Let $\tau,\rho\in \mathtt{A}$ and $\iota\in \mathcal{I}$. We have to show that
$(\tau\rho)\cdot v_{\iota}=\tau\cdot(\rho\cdot v_{\iota})$. Using associativity of $\circ$ and
involutivity of $\star$ we have
\begin{displaymath}
(\tau\circ \rho)\circ\iota\circ(\tau\circ \rho)^{\star}=
\tau\circ(\rho\circ\iota\circ\rho^{\star})\circ\tau^{\star}=:\alpha.
\end{displaymath}
This implies that both $(\tau\rho)\cdot v_{\iota}$ and $\tau\cdot(\rho\cdot v_{\iota})$ are
scalar multiplies of the same basis vector. 

Now let us check that the corresponding scalars coincide. Both scalars are non-zero if and only if
$\mathbf{r}(\alpha)=\mathbf{r}(\iota)$, so from now on we assume the latter equality.
Both scalars are powers of $\delta$ up to sign.
The powers of $\delta$ agree because of the fact that the correction term $\delta^{c(\tau,\iota)}$ 
does define an associative algebra structure
on $\mathbf{A}$. So, the only thing we are left to check is the fact that
\begin{equation}\label{eq73}
(-1)^{i(\pi_{\tau\circ\rho\circ \iota}\pi_{\iota},\pi_{\iota})} =
(-1)^{i(\pi_{\tau\circ (\rho\circ\iota\circ\rho^{\star})}
\pi_{(\rho\circ\iota\circ\rho^{\star})},\pi_{(\rho\circ\iota\circ\rho^{\star})})} 
(-1)^{i(\pi_{\rho\circ \iota}\pi_{\iota},\pi_{\iota})} .
\end{equation}
The fact that \eqref{eq55} defines a module structure over the symmetric group
reduces verification of \eqref{eq73} to the following identity:
\begin{equation}\label{eq753}
\pi_{\tau\circ\rho\circ\iota}\pi_{\iota}=
\pi_{\tau\circ(\rho\circ\iota\circ\rho^{\star})}\pi_{(\rho\circ\iota\circ\rho^{\star})}
\pi_{\rho\circ\iota}\pi_{\iota}.
\end{equation}
Let $A_1,A_2,\dots,A_{\mathbf{r}(\iota)}$ denote intersections of propagating lines of $\iota$ with
$\underline{n}$ ordered such that for all $1\leq i<j\leq \mathbf{r}(\iota)$ we have
$\min\{s\in A_i\}<\min\{s\in A_j\}$. Let $B_1,B_2,\dots,B_{\mathbf{r}(\iota)}$ denote
intersections of propagating lines of $\rho\circ\iota$ with $\underline{n}'$ ordered in the similar
way and, finally, let $C_1,C_2,\dots,C_{\mathbf{r}(\iota)}$ denote
intersections of propagating lines of $\tau\circ\rho\circ\iota$ with $\underline{n}'$ ordered in 
the similar way. Then, by the definition of $\pi$, the propagating line of $\tau\circ\rho\circ\iota$
containing $A_i$ also contains $C_{\pi_{\tau\circ\rho\circ\iota}(i)}$. Similarly, 
the propagating line of $\rho\circ\iota$
containing $A_i$ also contains $B_{\pi_{\rho\circ\iota}(i)}$ and, finally,
the propagating line of $\tau\circ(\rho\circ\iota\circ\rho^{\star})$
containing the unprimed version of $B_i$ also contains 
$C_{\pi_{\tau\circ(\rho\circ\iota\circ\rho^{\star})}(i)}$.
This reduces the equality \eqref{eq753} to the following equality:
\begin{displaymath}
\tau\circ\rho\circ\iota\circ\iota=
\tau\circ(\rho\circ\iota\circ\rho^{\star})\circ(\rho\circ\iota\circ\rho^{\star})\circ
\rho\circ\iota\circ\iota.
\end{displaymath}
This last equality follows from 
\begin{displaymath}
(\rho\circ\iota\circ\rho^{\star})\circ(\rho\circ\iota\circ\rho^{\star})\circ \rho\circ\iota=
\rho\circ\iota,
\end{displaymath}
which is proved by comparing the left and right parts on both the left and the right hand sides
of the equality since both these sides have the same rank and the element 
$(\rho\circ\iota\circ\rho^{\star})\circ(\rho\circ\iota\circ\rho^{\star})$ is idempotent
(the latter is equivalent to the assertion that  $\pi$ sends $\star$-self-dual elements to
involutions). This completes the proof of the lemma.
\end{proof}

\begin{lemma}\label{lem02}
If $\mathbf{A}$ is semi-simple, then the $\mathbf{A}$-module $\mathbb{C}[\mathcal{I}]$
is a multiplicity free direct sum of all Specht $\mathbf{A}$-modules.
\end{lemma}

\begin{proof}
For $\mathcal{J}\in\mathrm{spec}(\mathbf{A})$ let $\mathcal{I}_{\mathcal{J}}$ denote the set of elements of 
$\mathcal{I}$ which are in $\mathcal{J}$. Set $k:=\mathbf{r}(\mathcal{J})$. 
From the definition we have that $\mathbb{C}[\mathcal{I}]$ decomposes
into a direct sum of $\mathbb{C}[\mathcal{I}_{\mathcal{J}}]$'th for ${\mathcal{J}}\in\mathrm{spec}(\mathbf{A})$. 
So, it is enough to show that 
$\mathbb{C}[\mathcal{I}_{\mathcal{J}}]$ is a multiplicity free direct sum of Specht modules
$\Delta_{\mathcal{L}_{\mathcal{J}}}(N)$ where $N$ runs through the set of isomorphism classes of
simple $G_{\mathcal{J}}$-modules. 
Note that, by construction, the dimension of $\mathbb{C}[\mathcal{I}_{\mathcal{J}}]$ equals $q_{\mathcal{J}}$ times the number
of involutions in $G_{\mathcal{J}}$ and hence  agrees with the dimension of the 
multiplicity free direct sum of Specht modules as above. Hence if $\mathbf{A}$ is any of the Temperley-Lieb
type algebras, then the claim is obvious. Similarly in the case $k=0,1$ for any $\mathbf{A}$.
So we are left with the cases $\mathbf{A}\in\{\mathbf{P}_n,\mathbf{B}_n,\mathcal{P}\mathbf{B}_n,
\mathcal{W}_{a,b}\mathbf{B}_n,\mathcal{W}_{a,b}\mathcal{P}\mathbf{B}_n\}$ and $k>1$.

Assume first that $\mathbf{A}=\mathbf{P}_n$. Consider the following element:
\begin{displaymath}
\tau=\{1,1'\}\cup\{2,2'\}\cup \dots\cup\{k-1,(k-1)'\}\cup\{k,k+1,\dots,n,k',(k+1)',\dots,n'\}
\end{displaymath}
(see an example in Figure~\ref{fig4}).
We have $\tau=\tau^{\star}$ and $\tau\circ\tau=\tau=\tau\tau$. Let $H$ denote the set of all
diagrams $\rho$ of rank $k$ satisfying $\tau\circ\rho=\rho$ and $\rho\circ\tau=\rho$. Then
$\rho_1\circ\rho_2=\rho_1\rho_2$ for all $\rho_1,\rho_2\in H$ and hence the linear span of $H$
is isomorphic to the group algebra of the symmetric group $S_k$ and $\tau$ is the identity element in $H$.
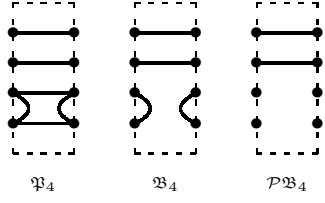
\begin{figure}
\special{em:linewidth 0.4pt} \unitlength 0.80mm
\begin{picture}(60.00,40.00)
%%%%%%%%%%%%%%%%%%%%%%%%%%%%%%%%%%%%%%%%%%%%%%%%%%%%%%%%
\put(05.00,15.00){\makebox(0,0)[cc]{$\bullet$}}
\put(05.00,20.00){\makebox(0,0)[cc]{$\bullet$}}
\put(05.00,25.00){\makebox(0,0)[cc]{$\bullet$}}
\put(05.00,30.00){\makebox(0,0)[cc]{$\bullet$}}
\put(15.00,15.00){\makebox(0,0)[cc]{$\bullet$}}
\put(15.00,20.00){\makebox(0,0)[cc]{$\bullet$}}
\put(15.00,25.00){\makebox(0,0)[cc]{$\bullet$}}
\put(15.00,30.00){\makebox(0,0)[cc]{$\bullet$}}
\put(25.00,15.00){\makebox(0,0)[cc]{$\bullet$}}
\put(25.00,20.00){\makebox(0,0)[cc]{$\bullet$}}
\put(25.00,25.00){\makebox(0,0)[cc]{$\bullet$}}
\put(25.00,30.00){\makebox(0,0)[cc]{$\bullet$}}
\put(35.00,15.00){\makebox(0,0)[cc]{$\bullet$}}
\put(35.00,20.00){\makebox(0,0)[cc]{$\bullet$}}
\put(35.00,25.00){\makebox(0,0)[cc]{$\bullet$}}
\put(35.00,30.00){\makebox(0,0)[cc]{$\bullet$}}
\put(45.00,15.00){\makebox(0,0)[cc]{$\bullet$}}
\put(45.00,20.00){\makebox(0,0)[cc]{$\bullet$}}
\put(45.00,25.00){\makebox(0,0)[cc]{$\bullet$}}
\put(45.00,30.00){\makebox(0,0)[cc]{$\bullet$}}
\put(55.00,15.00){\makebox(0,0)[cc]{$\bullet$}}
\put(55.00,20.00){\makebox(0,0)[cc]{$\bullet$}}
\put(55.00,25.00){\makebox(0,0)[cc]{$\bullet$}}
\put(55.00,30.00){\makebox(0,0)[cc]{$\bullet$}}
%%%%%%%%%%%%%%%%%%%%%%%%%%%%%%%%%%%%%%%%%%%%%%%%%%
\dashline{1}(05.00,10.00)(05.00,35.00)
\dashline{1}(15.00,35.00)(05.00,35.00)
\dashline{1}(15.00,35.00)(15.00,10.00)
\dashline{1}(05.00,10.00)(15.00,10.00)
\dashline{1}(25.00,10.00)(25.00,35.00)
\dashline{1}(35.00,35.00)(25.00,35.00)
\dashline{1}(35.00,35.00)(35.00,10.00)
\dashline{1}(25.00,10.00)(35.00,10.00)
\dashline{1}(45.00,10.00)(45.00,35.00)
\dashline{1}(55.00,35.00)(45.00,35.00)
\dashline{1}(55.00,35.00)(55.00,10.00)
\dashline{1}(45.00,10.00)(55.00,10.00)
%%%%%%%%%%%%%%%%%%%%%%%%%%%%%%%%%%%%%%%%%%%%%%%%%%
\thicklines
\drawline(05.00,20.00)(15.00,20.00)
\drawline(05.00,15.00)(15.00,15.00)
\drawline(05.00,30.00)(15.00,30.00)
\drawline(05.00,25.00)(15.00,25.00)
\drawline(25.00,30.00)(35.00,30.00)
\drawline(25.00,25.00)(35.00,25.00)
\drawline(45.00,30.00)(55.00,30.00)
\drawline(45.00,25.00)(55.00,25.00)
\qbezier(5,15)(10,17.50)(5,20)
\qbezier(15,15)(10,17.50)(15,20)
\qbezier(25,15)(30,17.50)(25,20)
\qbezier(35,15)(30,17.50)(35,20)
%%%%%%%%%%%%%%%%%%%%%%%%%%%%%%%%%%%%%%%%%%%%%%%%%%
\put(10,5){\makebox(0,0)[cc]{\tiny $\mathfrak{P}_4$}}
\put(30,5){\makebox(0,0)[cc]{\tiny $\mathfrak{B}_4$}}
\put(50,5){\makebox(0,0)[cc]{\tiny $\mathcal{P}\mathfrak{B}_4$}}
\end{picture}
%}}}
\caption{The element $\tau$ from the proof of Lemma~\ref{lem02} for $n=4$ and $k=2$}
\label{fig4}
\end{figure}

Without loss of generality we may assume that $H\subset\mathcal{L}_{\mathcal{J}}$. In this case from the definition 
of $\Delta_{\mathcal{L}_{\mathcal{J}}}(N)$ we may identify the image of the linear operator defined by 
$\tau$ acting on $\Delta_{\mathcal{L}_{\mathcal{J}}}(N)$ with $N$. Now, comparing \eqref{eq1}
and \eqref{eq55} we see that the image of the linear operator defined by 
$\tau$ acting on $\mathbb{C}[\mathcal{I}_{\mathcal{J}}]$ can be identified with
a Gelfand model for $H$. It follows that $\mathbb{C}[\mathcal{I}_{\mathcal{J}}]$ is a multiplicity free direct sum of 
Specht modules $\Delta_{\mathcal{L}_{\mathcal{J}}}(N)$ where $N$ corresponds to 
$\lambda\vdash k$.

If $\mathbf{A}=\mathbf{B}_n$, then we consider the following diagram (see an example in Figure~\ref{fig4}):
\begin{displaymath}
\tau=\{1,1'\}\cup\dots\cup\{k,k'\}\cup\{k+1,k+2\}\cup
\{(k+1)',(k+2)'\}\cup\dots\cup\{n-1,n\}\cup\{(n-1)',n'\}.
\end{displaymath}
Let $l:=(n-k)/2$ (note that $n$ and $k$ have the same parity). Let $H$ denote the set of all
diagrams $\rho$ of rank $k$ satisfying $\tau\circ\rho=\rho$ and $\rho\circ\tau=\rho$. 
Then the linear span of $H$ is again isomorphic to the group algebra of the symmetric group $S_k$ 
(acting on $\{1,2,\dots,k\}$) by sending a diagram to $\delta^{l}$ times the corresponding permutation
in $S_k$ (recall that $\delta\neq 0$ by our assumptions). 
Now the proof is completed in the same way as in the case $\mathbf{A}=\mathbf{P}_n$.

Finally, for $\mathbf{A}=\mathcal{P}\mathbf{B}_n$ we consider the following element:
\begin{displaymath}
\tau=\{1,1'\}\cup\dots\cup\{k,k'\}\cup\{k+1\}\cup
\{(k+1)'\}\cup\dots\cup\{n\}\cup\{n'\}
\end{displaymath}
(again, see an example in Figure~\ref{fig4}). Let $H$ denote the set of all
diagrams $\rho$ of rank $k$ satisfying $\tau\circ\rho=\rho$ and $\rho\circ\tau=\rho$. 
Then the linear span of $H$ is again isomorphic to the group algebra of the symmetric group $S_k$ 
(acting on $\{1,2,\dots,k\}$) by sending a diagram to $\delta^{n-k}$ times the corresponding permutation
in $S_k$. Now the proof is completed in the same way as in the case $\mathbf{A}=\mathbf{P}_n$.

The cases $\mathbf{A}=\mathcal{W}_{a,b}\mathbf{B}_n,\mathcal{W}_{a,b}\mathcal{P}\mathbf{B}_n$
are dealt with similarly and left to the reader.
\end{proof}

\begin{remark}\label{rem341}
{\rm  
The $\mathbf{A}$-module $\mathbb{C}[\mathcal{I}]$ makes also sense in the non-semisimple case in which
it is just isomorphic to a direct sum of Specht $\mathbf{A}$-modules. The latter are the cell modules
with respect to the cellular structure of $\mathbf{A}$. Maybe this is a sensible notion
of Gelfand model for a cellular algebras?
}
\end{remark}

\subsection{Example}\label{s3.3}

Let $\mathbf{A}=\mathcal{P}\mathbf{B}_2$. The set $\mathcal{I}$ consists of the six 
diagrams shown in Figure~\ref{fig3}. The algebra $\mathbf{A}$ is generated by 
$\alpha:=\iota_2$, $\beta:=\iota_3$ and $\gamma:=\iota_5$. The action of these generators 
(denoted by $\pi$ in the table below) on the basis elements of $\mathcal{I}$ is given by:
\begin{displaymath}
\begin{array}{c||c|c|c|c|c|c}
\pi\setminus\iota & \iota_1&\iota_2&\iota_3&\iota_4&\iota_5&\iota_6\\ 
\hline\hline
\alpha&\iota_1&-\iota_2&\iota_4&\iota_3&\iota_5&\iota_6\\
\hline
\beta&0&0&\iota_3&0&\iota_6&\delta\iota_6\\
\hline
\gamma&0&0&0&0&\delta\iota_5&\delta\iota_5\\
\end{array}
\end{displaymath}
The module $\mathbb{C}[\mathcal{I}]$ is a direct sum of four simple $\mathbf{A}$-modules, namely,
the linear spans of $\{\iota_1\}$, $\{\iota_2\}$, $\{\iota_3,\iota_4\}$ and $\{\iota_5,\iota_6\}$.
\begin{figure}
\special{em:linewidth 0.4pt} \unitlength 0.80mm
\begin{picture}(120.00,25.00)
%%%%%%%%%%%%%%%%%%%%%%%%%%%%%%%%%%%%%%%%%%%%%%%%%%%%%%%%
\put(05.00,10.00){\makebox(0,0)[cc]{$\bullet$}}
\put(05.00,15.00){\makebox(0,0)[cc]{$\bullet$}}
\put(15.00,10.00){\makebox(0,0)[cc]{$\bullet$}}
\put(15.00,15.00){\makebox(0,0)[cc]{$\bullet$}}
\put(25.00,10.00){\makebox(0,0)[cc]{$\bullet$}}
\put(25.00,15.00){\makebox(0,0)[cc]{$\bullet$}}
\put(35.00,10.00){\makebox(0,0)[cc]{$\bullet$}}
\put(35.00,15.00){\makebox(0,0)[cc]{$\bullet$}}
\put(45.00,10.00){\makebox(0,0)[cc]{$\bullet$}}
\put(45.00,15.00){\makebox(0,0)[cc]{$\bullet$}}
\put(55.00,10.00){\makebox(0,0)[cc]{$\bullet$}}
\put(55.00,15.00){\makebox(0,0)[cc]{$\bullet$}}
\put(65.00,10.00){\makebox(0,0)[cc]{$\bullet$}}
\put(65.00,15.00){\makebox(0,0)[cc]{$\bullet$}}
\put(75.00,10.00){\makebox(0,0)[cc]{$\bullet$}}
\put(75.00,15.00){\makebox(0,0)[cc]{$\bullet$}}
\put(85.00,10.00){\makebox(0,0)[cc]{$\bullet$}}
\put(85.00,15.00){\makebox(0,0)[cc]{$\bullet$}}
\put(95.00,10.00){\makebox(0,0)[cc]{$\bullet$}}
\put(95.00,15.00){\makebox(0,0)[cc]{$\bullet$}}
\put(105.00,10.00){\makebox(0,0)[cc]{$\bullet$}}
\put(105.00,15.00){\makebox(0,0)[cc]{$\bullet$}}
\put(115.00,10.00){\makebox(0,0)[cc]{$\bullet$}}
\put(115.00,15.00){\makebox(0,0)[cc]{$\bullet$}}
%%%%%%%%%%%%%%%%%%%%%%%%%%%%%%%%%%%%%%%%%%%%%%%%%%
\dashline{1}(05.00,05.00)(05.00,20.00)
\dashline{1}(15.00,20.00)(05.00,20.00)
\dashline{1}(15.00,20.00)(15.00,05.00)
\dashline{1}(05.00,05.00)(15.00,05.00)
\dashline{1}(25.00,05.00)(25.00,20.00)
\dashline{1}(35.00,20.00)(25.00,20.00)
\dashline{1}(35.00,20.00)(35.00,05.00)
\dashline{1}(25.00,05.00)(35.00,05.00)
\dashline{1}(45.00,05.00)(45.00,20.00)
\dashline{1}(55.00,20.00)(45.00,20.00)
\dashline{1}(55.00,20.00)(55.00,05.00)
\dashline{1}(45.00,05.00)(55.00,05.00)
\dashline{1}(65.00,05.00)(65.00,20.00)
\dashline{1}(75.00,20.00)(65.00,20.00)
\dashline{1}(75.00,20.00)(75.00,05.00)
\dashline{1}(65.00,05.00)(75.00,05.00)
\dashline{1}(85.00,05.00)(85.00,20.00)
\dashline{1}(95.00,20.00)(85.00,20.00)
\dashline{1}(95.00,20.00)(95.00,05.00)
\dashline{1}(85.00,05.00)(95.00,05.00)
\dashline{1}(105.00,05.00)(105.00,20.00)
\dashline{1}(115.00,20.00)(105.00,20.00)
\dashline{1}(115.00,20.00)(115.00,05.00)
\dashline{1}(105.00,05.00)(115.00,05.00)
%%%%%%%%%%%%%%%%%%%%%%%%%%%%%%%%%%%%%%%%%%%%%%%%%%
\thicklines
\drawline(05.00,10.00)(15.00,10.00)
\drawline(05.00,15.00)(15.00,15.00)
\drawline(25.00,10.00)(35.00,15.00)
\drawline(25.00,15.00)(35.00,10.00)
\drawline(45.00,10.00)(55.00,10.00)
\drawline(65.00,15.00)(75.00,15.00)
\qbezier(85,15)(90,12.50)(85,10)
\qbezier(95,15)(90,12.50)(95,10)
%%%%%%%%%%%%%%%%%%%%%%%%%%%%%%%%%%%%%%%%%%%%%%%%%%
\put(10,1){\makebox(0,0)[cc]{\tiny $\iota_1$}}
\put(30,1){\makebox(0,0)[cc]{\tiny $\iota_2$}}
\put(50,1){\makebox(0,0)[cc]{\tiny $\iota_3$}}
\put(70,1){\makebox(0,0)[cc]{\tiny $\iota_4$}}
\put(90,1){\makebox(0,0)[cc]{\tiny $\iota_5$}}
\put(110,1){\makebox(0,0)[cc]{\tiny $\iota_6$}}
\end{picture}
%}}}
\caption{$\star$-self-dual diagrams for $\mathcal{P}\mathbf{B}_2$}
\label{fig3}
\end{figure}
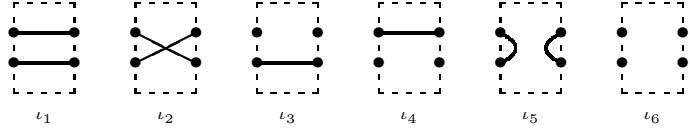

\vspace{0.5cm}

\noindent
V.M.:\hspace{4mm} Department of Mathematics, Uppsala University, SE 47106,
Uppsala, SWEDEN, e-mail: {\tt mazor\symbol{64}math.uu.se}

\begin{thebibliography}{99999}
\bibitem[APR1]{APR} R.~Adin, A.~Postnikov, Y.~Roichman. Combinatorial Gelfand models. 
J. Algebra {\bf 320} (2008), no. 3, 1311--1325.
\bibitem[APR2]{APR2} R.~Adin, A.~Postnikov, Y.~Roichman. A Gelfand model for wreath 
products. Israel J. Math. {\bf 179} (2010), 381--402.
\bibitem[Ar]{Ar} J.~Araujo. A Gelʹfand model for a Weyl group of type $B_n$. 
Beitr{\"a}ge Algebra Geom. {\bf 44} (2003), no. 2, 359--373.
\bibitem[ABi]{ABi} J.~Araujo, J.~Bige{\'o}n. A Gelʹfand model for the symmetric 
generalized group. Comm. Algebra {\bf 37} (2009), no. 5, 1808--1830.
\bibitem[ABr]{AB} J.~Araujo, T.~Bratten. Gelfand models for classical Weyl groups.
Preprint arXiv:1112.3585. 
\bibitem[BH]{BH} G.~Benkart, T.~Halverson. Motzkin Algebras.
Preprint  arXiv:1106.5277.
\bibitem[BGG]{BGG} J.~Bernstein, I.~Gelfand, S.~Gelfand. Models of representations 
of compact Lie groups, Funktsional. Anal. i Prilozhen. {\bf 9} (1975) 61--62.
\bibitem[Br]{Br} R.~Brauer. On algebras which are connected with the semisimple 
continuous groups.  Ann. of Math. (2)  {\bf 38}  (1937),  no. 4, 857--872.
\bibitem[Ca]{Ca}  F.~Caselli. Involutory reflection groups and their models. 
J. Algebra {\bf 324} (2010), no. 3, 370--393.
\bibitem[CF1]{CF}  F.~Caselli, R.~Fulci. Refined Gelfand models for wreath products. 
European J. Combin. {\bf 32} (2011), no. 2, 198--216.
\bibitem[CF2]{CF2} F.~Caselli, R.~Fulci. Gelfand models and Robinson-Schensted 
correspondence. J. Algebraic Combin. {\bf 36} (2012), no. 2, 175--207.
\bibitem[GO]{GO} S.~Garge, J.~Oesterl{\'e}. On Gelfand models for finite Coxeter groups. 
J. Group Theory {\bf 13} (2010), no. 3, 429--439.
\bibitem[GM]{GM} O.~Ganyushkin, V. Mazorchuk, Classical finite 
transformation semigroups, An introduction. Algebra and Applications, 
Vol. {\bf 9}, Springer Verlag, to appear in 2008. 
\bibitem[GMS]{GMS} O.~Ganyushkin, V. Mazorchuk, B.~Steinberg,
On the irreducible representations of a finite semigroup, 
preprint math.RT arxiv:0712.2076.
\bibitem[Ha]{Ha} T.~Halverson. Characters of the centralizer algebras of mixed tensor 
representations of $GL(r,\mathbb{C})$ and the quantum
group $U_q(\mathfrak{gl}(r,\mathbb{C}))$. Pacific J. Math. {\bf 174} (1996), 359--410.
\bibitem[IRS]{IRS} N.~Inglis, R.~Richardson, J.~Saxl. An explicit model for 
the complex representations of $S\sb n$.  Arch. Math. (Basel)  {\bf 54} 
(1990),  no. 3, 258--259.
\bibitem[Jo]{Jo} V.~Jones. The Potts Model and the symmetric group, in Subfactors: 
Proceedings of the Tanaguchi Symposium on Operator Algebras, Kyuzeso, 1993, pp. 259--267, 
World Scientific, River Edge, NJ 1994.
\bibitem[KV]{KV} V.~Kodiyalam, D.-N.~Verma. A natural representation model 
for symmetric groups. Preprint arXiv:math/0402216.
\bibitem[Ko]{Ko} K.~Koike. On the decomposition of tensor products of the 
representations of classical groups: by means of universal
characters. Adv. Math. {\bf 74} (1989), 57--86.
\bibitem[KM]{KM} G.~Kudryavtseva, V.~Mazorchuk. Combinatorial Gelfand models 
for some semigroups and q-rook monoid algebras. Proc. Edinb. Math. Soc. (2) 
{\bf 52} (2009), no. 3, 707--718.
\bibitem[Mar1]{Mar} P.~Martin. Temperley-Lieb algebras for nonplanar 
statistical  mechanics---the partition algebra construction.  
J. Knot Theory Ramifications  {\bf 3}  (1994),  no. 1, 51--82.
\bibitem[Mar2]{Mar2} P.~Martin. Potts models and related problems 
in statistical mechanics. Series on Advances in Statistical Mechanics, 
{\bf 5}. World Scientific Publishing Co., Inc., Teaneck, NJ, 1991.
\bibitem[MM1]{MM1} P.~Martin, V. Mazorchuk. Partitioned binary relations. Preprint 
arXiv:1102.0862, to appear in Math. Scand.
\bibitem[MM2]{MM2} P.~Martin, V. Mazorchuk. On the representation theory of partial Brauer 
algebras. Preprint arXiv:1205.0464. To appear in Quart. J. Math.
\bibitem[Maz1]{Ma1} V. Mazorchuk. On the structure of Brauer semigroup 
and its partial analogue.  Problems in Algebra \textbf{13} (1998), 29--45.
\bibitem[Maz2]{Ma2} V. Mazorchuk. Endomorphisms of $\mathfrak{B}_n$, $\mathcal{P}\mathfrak{B}_n$  
and  $\mathfrak{C}_n$. Comm. Algebra \textbf{30} (2002), no.~7, 3489--3513. 
\bibitem[Ni]{Ni} P.~Nikitin. The centralizer algebra of the diagonal action 
of the group $GL_n(\mathbb{C})$ in a mixed tensor space. J. Math.
Sci. \textbf{141} (2007), 1479--1493.
\bibitem[Ru]{Ru} H.~Rui. A criterion on the semisimple brauer algebras. J. Comb. Theory
Ser. A \textbf{111} (2005), no. 1, 78--88.
\bibitem[Sa]{Sa} 
B.~Sagan, The symmetric group. Representations, combinatorial algorithms, 
and symmetric functions. Second edition. Graduate Texts in Mathematics, 
{\bf 203}. Springer-Verlag, New York, 2001.
\bibitem[Sch]{Sch} M.~Schramm. Combinatorial Gelfand models. Master Thesis. 
Uppsala University, 2012.
\bibitem[TL]{TL} H.~Temperley, E.~Lieb. Relations between the 
``percolation'' and ``colouring'' problem and other graph-theoretical 
problems associated with regular planar lattices: some exact results for 
the ``percolation'' problem. Proc. Roy. Soc. London Ser. A {\bf 322} 
(1971), no. 1549, 251--280.
\bibitem[Tu]{Tu} V.~Turaev. Operator invariants of matrices and $R$-matrices. 
Izv. Akad. Nauk SSSR {\bf 53} (1989), 1073--1107.
\end{thebibliography}
\end{document}